\documentclass{amsart}

\usepackage{amsthm}
\usepackage{amssymb}
\usepackage{latexsym}
\usepackage{amsfonts}
\usepackage{amsmath}
\usepackage{amstext}
\usepackage[numbers]{natbib}
\newtheorem{theorem}{Theorem}
\newtheorem{lemma}[theorem]{Lemma}
\newtheorem{corollary}[theorem]{Corollary}
\newcommand{\RR}{\mathbb{R}}

\newcommand{\ip}[2]{\ensuremath{\left\langle #1, #2 \right\rangle}}
\newcommand{\deriv}[3]{\ensuremath{\frac{#1 #2}{#1 #3}}}
\newcommand{\pd}[2][1]{\deriv{\partial}{#1}{#2}}
\newcommand{\abs}[1]{\ensuremath{\left|#1\right|}}
\newcommand{\avg}[1]{\ensuremath{\overline{#1}}}
\newcommand{\E}{\mathrm{e}}
\newcommand{\nor}[1][\empty]{\vec{N}_{#1}}
\newcommand{\tang}[1][\empty]{\vec{T}_{#1}}
\newcommand{\curv}{\ensuremath{k}}
\newcommand{\chord}[2]{\ensuremath{\overline{{#1}{#2}}}}

\begin{document}

\title[Grayson's theorem by distance comparison]{Curvature bound for
curve shortening flow via distance comparison and a direct proof of
Grayson's theorem} 
\author{Ben Andrews}
\address{MSI, 
	             ANU, 
	             ACT 0200, Australia}
	             \email{Ben.Andrews@anu.edu.au}
\author{Paul Bryan}
\address{MSI, 
	             ANU, 
	             ACT 0200, Australia}
	             \email{Paul.Bryan@anu.edu.au}
\begin{abstract}
A new isoperimetric estimate is proved for embedded closed curves evolving by curve shortening flow, normalized to have total length $2\pi$.  The estimate bounds the length of any chord from below in terms of the arc length between its endpoints and elapsed time.
Applying the estimate to short segments we deduce directly that the maximum curvature decays exponentially to $1$. This gives a self-contained 
proof of Grayson's theorem which does not require the monotonicity
formula or the classification of singularities.  
\end{abstract}

\maketitle

\section{Introduction}
\label{sec:intro}

Under the curve shortening flow, a curve in $\RR^2$ moves in its normal direction with speed given by its curvature.  More precisely, if $\tilde F_0:\ S^1\to\RR^2$ is a smooth immersion, we are 
concerned with the solution $\tilde F:\ [0,\tilde T)\to\RR^2$ of the initial value problem
\begin{align}
  \pd[\tilde{F}]{\tau} &= -\curv \nor\label{eq:csf_ivp} \\
  \tilde{F}(\cdot, 0) &= \tilde F_0(\cdot)\notag
\end{align}
where $\curv$ is the curvature with respect to the unit
normal $\nor$.  It is well known \cite[Section 2]{MR840401} that there exists a unique solution 
$\tilde F:\ S^1\times[0,\tilde T)\to\RR^2$ such that $\tilde F(.,\tau)$ is a smooth immersion for each $\tau$, and the maximum curvature $\curv_{\max}(\tau)=\max\{|\curv(s,\tau)|:\ s\in S^1\}$ becomes unbounded as $\tau$ approaches the maximal time $\tilde T$.

Following work by Gage \cite{MR726325}, \cite{MR742856} and Gage and Hamilton
\cite{MR840401} on the solution of \eqref{eq:csf_ivp} for convex curves, Grayson \cite{MR906392} proved that any embedded closed curve evolves to become convex, and subsequently shrinks to a point while becoming circular in shape.  His argument was rather delicate, requiring separate analyses of what may happen under various geometric configurations, and special arguments in each case to show that the curve must indeed become convex.   More recently the proof has been simplified by using isoperimetric estimates to rule out certain kinds of behaviour:  Huisken \cite{MR1656553} gave an isoperimetric estimate relating chord length to arc length, and Hamilton \cite{MR1369140} gave an estimate controlling the ratio of the isoperimetric profile to that of a circle of the same area.   Either of these arguments can be used to deduce Grayson's theorem, by making use of previous results concerning the classification of singularities:   If the maximum curvature remains comparable to that of a circle with the same extinction time, Huisken's monotonicity formula \cite{MR1030675} implies that the curve has asymptotic shape given by a self-similar solution of curve-shortening flow, which by the classification of Abresch and Langer \cite{MR845704} must be a circle.   Otherwise, one can use a blow-up procedure to produce a convex limiting curve to which Hamilton's Harnack estimate \cite{MR1316556} for the curve-shortening flow can be applied to show that it is a `grim reaper' curve (see for example the argument given in \cite[Section 8]{MR1131441}).  But this implies that the isoperimetric bound must be violated, proving Grayson's theorem.

Our purpose in this paper is to give a proof of Grayson's theorem which does not require any of the additional machinery described above:  By refining the isoperimetric argument of Huisken, we obtain stronger control on the chord distances, sufficient to imply a curvature bound.  The curvature bound we obtain is remarkably strong, and immediately implies that after rescaling the evolving curves to have length $2\pi$ the maximum curvature approaches $1$ at a sharp rate.  The convergence of the rescaled curves to circles is then straightforward.

Our result is most easily formulated in terms of a normalized flow, which we now introduce:  Given a solution $\tilde F$ of \eqref{eq:csf_ivp}, we define $F:\ S^1\times[0,T)\to\RR^2$ by
$$
F(p,t)=\frac{2\pi}{L[\tilde F(.,\tau)]}\tilde F(p,\tau),
$$
where
$$
t = \int_0^\tau \left(\frac{2\pi}{L[\tilde F(.,\tau')]}\right)^2\,d\tau',\qquad\text{and}\qquad
T=\int_0^{\tilde T}\left(\frac{2\pi}{L[\tilde F(.,\tau')]}\right)^2\,d\tau'.
$$
Then $L[F(.,t)]=2\pi$ for every $t$, and $F$ evolves according to the normalized equation
\begin{equation}
  \label{eq:csf_rescaled}
  \pd[F]{t} = \avg{\curv^2} F - \curv \nor
\end{equation}
where $\curv$ denotes the curvature of the normalized curve $F$, and we introduced the average curvature $\avg{\curv^2} =
(2\pi)^{-1}\int_{S^1}\curv^2$.  Our main result is an isoperimetric bound for embedded curves evolving by Equation \eqref{eq:csf_rescaled}, controlling lengths of chords in terms of the arc length between their endpoints and elapsed time.

\section{Distance comparison for smooth embedded curves}
\label{sec:distcomp}

We denote the chord length by $d(p,q,t)=|F(q,t)-F(p,t)|$, and the arc length along the curve $F(.,t)$ by $\ell(p,q,t)$.
Our main result is the following:
\begin{theorem}
\label{thm:distcomp}
Let $F:\ S^1\times[0,T)\to\RR^2$ be a smooth embedded solution of
the normalised curve-shortening flow \eqref{eq:csf_rescaled} with fixed total length $2
\pi$. Then there exists $\bar{t}\in\RR$ such that for every $p$ and $q$ in
$S^1$ and every $t\geq 0$,
\begin{equation}\label{eq:d.f.ineq}
  d(p,q,t) \geq 
  f\left(\ell(p,q,t),t-\bar{t}\right),
\end{equation}
where $f$ is defined by
$f(x,t) = 2\E^t\arctan\left(\E^{-t}\sin\left(\frac{x}{2}\right)\right)$ for $t\in\RR$ and $x\in[0,2\pi]$.
\end{theorem}

\begin{proof}
We begin by proving that for any smooth embedded closed curve $F_0$ the inequality $d\geq f(\ell,-\bar{t})$ holds for sufficiently large $\bar{t}$.  In particular this implies there exists $\bar{t}\in\RR$ such that the inequality \eqref{eq:d.f.ineq} is satisfied at $t=0$.  First we compute
\begin{equation*}
  \pd[]{t}f(x,t) =
 2\E^t\left[\arctan(\E^{-t}\sin(x/2))-\frac{\E^{-t}\sin(x/2)}{(1+\E^{-2t}\sin^2(
  x/2))}\right] = 2\E^tg(\E^{-t}\sin(x/2)),
\end{equation*}
where $g(z) = \arctan z -\frac{z}{1+z^2}$.  Then $g(0)=0$ and 
$
  g'(z)=\frac{2z^2}{(1+z^2)^2}>0
$
for $z>0$, so $g(z)>0$, and $f$ is strictly
increasing in $t$.  Also note that $\lim_{t\to\infty}f(x,t)=2\sin(x/2)$, and $\lim_{t\to-\infty}f(x,t) = 0$.  Define $a(p,q)=\inf\{e^t:\ d(p,q)\geq f(\ell(p,q),-t)\}$ for $p\neq q$ in $S^1$.  Then by the implicit function theorem $a$ is continuous, and smooth and positive where $0<d<2\sin(\ell/2)$ where it is defined by the identity
\begin{equation}\label{eq:defa}
d(p,q) = f(\ell(p,q),-\log(a(p,q))).
\end{equation}

\begin{lemma}
\label{lem:a0}
The function $a$ extends to a continuous function on $S^1\times S^1$ by defining
$$
a(p,p)=\sqrt{\frac{\max\{\curv(p)^2-1,0\}}{2}}.
$$
In particular $\bar{a} =
\sup\left\{a(p,q):\ p\neq q\right\}$ is finite.
\end{lemma}

\begin{proof}
We must show that $a$ is continuous at each point $(p,p)$.
Fix $p$ and parametrise by arc length $s$ so that $F_0(0) = p_0$.  The Taylor expansion of $F_0$ about $s_1$ gives
\begin{align*}
F_0(s_2)-F_0(s_1) &= (s_2-s_1)\tang(s_1)-\frac{(s_2-s_1)^2}{2}\curv(s_1)\nor(s_1)\\
&\quad\null
-\frac{(s_2-s_1)^3}{6}\left(\curv_s(s_1)\nor(s_1)+\curv(s_1)^2\tang(s_1)\right)+o(|s_2-s_1|^4).
\end{align*}
Computing the squared length of this we find
\begin{align*}
d(s_1,s_2)^2&=|s_2-s_1|^2\left(1-\frac{(s_2-s_1)^2}{12}\curv(s_1)^2+O(|s_2-s_1|^3)\right)\\
&=|s_2-s_1|^2\left(1-\frac{(s_2-s_1)^2}{12}\curv(0)^2+O((|s_2|+|s_1|)|s_2-s_1|^2)\right).
\end{align*}
Since $\ell(s_2,s_1)=|s_2-s_1|$, it follows that 
$$
d(s_1,s_2) = \ell(s_1,s_2)-\frac{\ell(s_1,s_2)^3}{24}\left(\curv(0)^2+O((|s_2|+|s_1|))\right).
$$
Now the Taylor expansion of $f$ about $x=0$ gives
$$
f(x,-\log a) = x-\frac{1+2a^2}{24}x^3+O(x^4),
$$
so since $2\sin(x/2) = x-\frac{1}{24}x^3+O(x^4)$, the identity \eqref{eq:defa} gives for $\curv(0)^2>1$ that 
$$
\ell-\left(\frac{\curv(0)^2}{24}+O(|s_1|+|s_2|)\right)\ell^3 = \ell-\frac{1+2a^2}{24}\ell^3+O(\ell^4),
$$
so that $\max\{\curv(0)^2,1\} = 1+2a(s_1,s_2)^2+O(|s_1|+|s_2|)$.
In particular we have 
\begin{equation}\label{eq:lima}
\lim_{(s_1,s_2)\to(0,0)}a(s_1,s_2) = \sqrt{\frac{\max\{\curv(0)^2-1,0\}}{2}},
\end{equation}
proving that $a$ is continuous.
\qed\end{proof}

By the construction of $a$ and the monotonicity of $f$ in $a$ we have
$$
d(p,q)\geq f(\ell(p,q),-\log a(p,q))\geq f(\ell(p,q),-\bar{t}),
$$
where $\bar{t}=\log \bar a$, 
so the inequality in the Theorem holds for $t=0$.

To show the result for positive times we use a maximum principle
argument.  Define $Z:\ S^1\times
S^1\times[0,T)\to \RR$ by
\begin{equation}
  \label{eq:z}
  Z(p,q,t) = d(p,q,t) -
  f\left(\ell(p,q,t),t-\bar{t}\right).
\end{equation}
Note that $Z$ is continuous on $S^1\times S^1\times[0,T)$ and smooth where $p\neq q$.
Fix $t_1\in(0,T)$, and choose $C>\sup\{\avg{\curv^2}(t):\ 0\leq t\leq t_1\}$
We prove by contradiction that $Z_{\varepsilon}= Z+\varepsilon \E^{Ct}$ remains positive on $S^1\times S^1\times [0,t_1]$ for any $\varepsilon>0$.  At $t=0$ and on the diagonal $\{(p,p): \ p\in S^1\}$ we have $Z_{\varepsilon}\geq \varepsilon>0$,
so if $Z_{\varepsilon}$ does not remain positive then there
exists $t_0\in(0,t_1]$ and $(p_0,q_0)\in S^1\times S^1$ with $p_0\neq q_0$ such that
$Z_{\varepsilon}(p_0,q_0,t_0) = 0=\inf\{Z_{\varepsilon}(p,q,t):\ p,q\in S^1,\ 0\leq t\leq t_0\}$.
It follows that at $(p_0,q_0,t_0)$ we have $Z=\varepsilon\E^{Ct_0}$, $\pd[Z]{t}+C\varepsilon\E^{Ct_0}=\frac{\partial Z_\varepsilon}{\partial t}\leq 0$, while the first spatial derivative of $Z$ vanishes and the second is non-negative.

We parametrize using the arc-length parameter at time $t_0$, and choose the normal $\nor$ to point out of the region enclosed by the curve.  For arbitrary real $\xi$ and $\eta$, let $\sigma(u)=(p_0+\xi u,q_0+\eta u,t_0)$.  Then we compute
\begin{equation}\label{eq:deriv1}
\frac{\partial}{\partial u}Z(\sigma(u))=\xi\left(-\ip{w}{\tang[p]} + f'\right) +
  \eta\left(\ip{w}{\tang[q]} -f'\right),
\end{equation}
where $f'$ denotes the derivative in the first argument, $\tang[p]=\frac{\partial F}{\partial s}(p,t)$, and we define for $p\neq q$
\begin{equation*}
  w(p,q,t) = \frac{F(q, t) - F(p, t)}{d(p, q,t)}.
\end{equation*}
The right-hand side of Equation \eqref{eq:deriv1} vanishes at $u=0$, so we have
\begin{equation}
  \label{eq:firstderiv}
 f' = \ip{w}{\tang[p_0]} =
  \ip{w}{\tang[q_0]}.
\end{equation}
There are two possibilities:   Either
$\tang[q_0] = \tang[p_0]\neq w$, or $w$ bisects $\tang[p_0]$ and $\tang[q_0]$.

We begin by ruling out the first case.  Since $\tang[p_0]=\tang[q_0]\neq w$, the normal makes an acute angle with the chord $\chord{p_0}{q_0}$ at one endpoint, and an obtuse angle at the other. Therefore points on the chord near one endpoint are inside the region, while points near the other endpoint are outside, implying that there is at least one other point where the curve $F(.,t_0)$ meets the chord.  We may assume that an
intersection occurs at $s$ with $p_0 < s < q_0$.  Then we have
\begin{align*}
  d(p_0,q_0) &= d(p_0,s) + d(s,q_0) \\
  \ell(p_0,q_0) &= \min\{\ell(p_0,s) + \ell(s,q_0),2\pi-\ell(p_0,s)-\ell(s,q_0)\}
\end{align*}
$f=f(.,a)$ is strictly concave, so $f(x+y)=f(x+y)+f(0)\leq f(x)+f(y)$ whenever $x,y>0$ and $x+y<2\pi$.  Noting also that $f(x)=f(2\pi-x)$, we have
\begin{align*}
  Z(p_0, q_0) &= d(p_0, q_0) -f\left(\ell(p_0, q_0),a\right) \\ 
  &= d(p_0,s)+d(s,q_0)-f\left(\ell(p_0,s)+\ell(s,q_0),a\right)\\
  &> d(p_0, s) - f\left(\ell(p_0, s), a\right) + d(s, q_0) - 
  f\left(\ell(s, q_0), a\right) \\
  &= Z(p_0, s) + Z(s, q_0)
\end{align*}
and so either $Z(p_0, s) < Z(p_0, q_0)$ or $Z(s, q_0) < Z(p_0, q_0)$, which is impossible.

Now let us consider the second case.  The second derivative of $Z$ along $\sigma$ is
\begin{equation*}
\begin{split}
  \frac{\partial^2}{\partial u^2}Z(\sigma(u))\Big|_{u=0}&= \xi^2\left[\frac{1}{d}\left(1 -
  \ip{w}{\tang[p_0]}^2\right) + \ip{w}{\curv_{p_0}\nor[p_0]} -
  f''\right] \\
  & + \eta^2\left[\frac{1}{d}\left(1 - \ip{w}{\tang[q_0]}^2\right) -
  \ip{w}{\curv_{q_0}\nor[q_0]} - f''\right] \\
  & + 2\xi\eta
  \left[\frac{1}{d}\left(\ip{w}{\tang[p_0]}\ip{w}{\tang[q_0]} -
  \ip{\tang[p_0]}{\tang[q_0]}\right) +
 f''\right]
\end{split}
\end{equation*}
Since $w$ bisects $\tang[p_0]$ and $\tang[q_0]$ we can write
$\ip{\tang[p_0]}{w}=\ip{\tang[q_0]}{w}\cos\theta$ and 
$\ip{\tang[p_0]}{\tang[q_0]} = 2\cos^2\theta-1$.  Choosing $\xi = 1$ and $\eta=-1$ then gives
\begin{equation}
  \label{eq:d2z}
  0 \leq \ip{w}{\curv_{p_0} \nor[p_0] - \curv_{q_0}
  \nor[q_0]} - 4f''.
\end{equation}

Under the rescaled flow equation \eqref{eq:csf_rescaled}, $d$ and $\ell$ evolve as follows:
\begin{align*}
  \pd[d]{t} &= \frac{1}{d}\ip{-\curv_{p_0} \nor[p_0] +
    \avg{\curv^2} F_{p_0} +\curv_{q_0} \nor[q_0] -
    \avg{\curv^2} F_{q_0}}{F_{p_0} - F_{q_0}} =
  \ip{w}{\curv_{p_0} -\curv_{q_0}} + \avg{\curv^2}d; \\
  \pd[\ell]{t} &= \avg{\curv^2} \ell - \int_{p_0}^{q_0} k^2 ds.
\end{align*}
The latter is obtained from Equation
\eqref{eq:csf_rescaled} as in \cite[Lemma 3.1.1]{MR840401} to compute
\begin{equation*}
  \pd[]{t}\abs{\pd[F]{p}} = (\avg{\curv^2} - \curv^2)\abs{\pd[F]{p}}.
\end{equation*}
From these we obtain an expression for the time derivative of $Z$:
\begin{align*}
  -C\varepsilon\E^{Ct_0}\geq \pd[Z]{t} &= \pd[d]{t} - f'\pd[\ell]{t} -\pd[f]{t} \\
  &= \ip{w}{\curv_{p_0} \nor[p_0] - \curv_{q_0} \nor[q_0]} +
  \avg{\curv^2}d - f' \left(\avg{\curv^2} \ell -
    \int_{p_0}^{q_0} \curv^2 ds\right)  -\pd[f]{t}  \\
  &= \ip{w}{\curv_{p_0} \nor[p_0] - \curv_{q_0}
    \nor[q_0]}+ \avg{\curv^2} (\varepsilon\E^{Ct_0} + f - f'\ell)
    + f'\int_{p_0}^{q_0} \curv^2 ds
   -\pd[f]{t} . 
\end{align*}
From equation \eqref{eq:d2z},
\begin{equation}
  \label{eq:dzdt}
  -C\varepsilon\E^{Ct_0} \geq 4f''+\avg{\curv^2} \left(\varepsilon\E^{Ct_0} + f - f'\ell\right)+ f'\int_{p_0}^{q_0} \curv^2 ds
  -\pd[f]{t} .
\end{equation}
Now we observe that since $f$ is concave, $(f-f'\ell)'=-f''\ell>0$, so $f-f'\ell>0$ for $\ell>0$.  We estimate the coefficient $\avg{\curv^2}$ of $f-f'\ell$ using H\"older's inequality, to give $\avg{\curv^2}\geq \left(\avg{\curv}\right)^2=1$, since $\int\curv ds=2\pi=\int\,ds$.  Since $\ell\leq\pi$ we also have $f'\geq 0$, so we can also estimate the second-last term in \eqref{eq:dzdt} using H\"older's inequality: 
$$
\int_{p_0}^{q_0}\curv^2ds\geq \frac{\left(\int_{p_0}^{q_0}|\curv|ds\right)^2}{\ell}
\geq\frac{\theta^2}{\ell},
$$
where $\theta$ is the angle between $\tang[p_0]$ and $\tang[q_0]$.  This is twice the angle between $\tang[p_0]$ and $w$, so by Equation \eqref{eq:firstderiv} we have
$\theta=2\arccos(f')$, and \eqref{eq:dzdt} becomes
$-C\varepsilon\E^{Ct_0}\geq Lf +\avg{\curv^2}\varepsilon\E^{Ct_0}$ and hence $Lf<0$ by our choice of $C$, where
$$
L f = 4f''+f-f'\ell + 4\frac{f'}{\ell}\left(\arccos(f')\right)^2 -\pd[f]{t} .
$$
We make one further estimation:  Observing that $z\mapsto h(z):=(\arccos(z))^2$ is a convex function on $[0,1]$ we estimate 
$$
h(f') \geq h(\cos(\ell/2)) + h'(\cos(\ell/2))(f'-\cos(\ell/2)) = \frac{\ell^2}{4}-\frac{\ell}{\sin(\ell/2)}(f'-\cos(\ell/2)).
$$
This gives $Lf\geq \tilde{L}f$, where
$$
\tilde{L}f = 4f''+f-\frac{4f'}{\sin(\ell/2)}\left(f'-\cos(\ell/2)\right) -\pd[f]{t} .
$$
Thus we have a contradiction if $\tilde{L}f\geq 0$, and $f$ is concave for each $t$.  We leave it to the reader to check that $f$ is in fact a solution of 
$\tilde{L}f=0$.   We remark that while our own discovery of the function $f$ was purely serendipitous, it could reasonably be produced by changing variable from $\ell$ to $\sin(\ell/2)$ and seeking a similarity solution of $\tilde Lf=0$. \end{proof}

\section{The curvature bound and long time existence}
\label{sec:curv}

\begin{theorem}
\label{thm:curv}
With $\bar{t}$ as in Theorem \ref{thm:distcomp}, we have
$$
\sup\{\curv(p,t)^2:\ p\in S^1\}\leq 1+2e^{-2(t-\bar{t})}
$$
for $0\leq t<T$.
\end{theorem}

\begin{proof}
By Lemma \ref{lem:a0} and Theorem \ref{thm:distcomp}, for $t\geq 0$ we have
for each $p\in S^1$ 
$$
\sqrt{\frac{\max\{\curv(p,t)^2-1,0\}}{2}} = a(p,p,t)\leq \sup\{a(p,q,t):\ p\neq q\}\leq \E^{\bar{t}-t}.
$$
\qed\end{proof}

\begin{corollary}\label{cor:LTE}
$T=\infty$, and $\left|\frac{\partial^n\curv}{\partial s^n}\right|\leq C(n,\bar{t})(1+t^{-n/2})$ for each $n>0$ and $t>0$.
\end{corollary}

\begin{proof}
Suppose $T<\infty$.  Theorem \ref{thm:curv} gives a bound on curvature of the form $|\curv(p,t)|\leq C$ for all $t\in[0,T]$.  The un-normalized equation can be recovered from the normalized one by
setting $\lambda(t) = \frac{L[\tilde F_0]}{2\pi}\exp\left(-\int_0^t\avg{\curv^2}(t')dt'\right)$, and defining
$\tilde F(p,\tau) = \lambda(t)F(p,t)$, where $\tau = \int_0^t \lambda(t')^2\,dt'$.  In particular, $\curv_{\max}(\tau)$ is bounded for $\tau\in[0,\tilde T)$, which is impossible.

The bounds follow from a standard bootstrapping argument.  For example, bounds on $\frac{\partial\curv}{\partial s}$ can be obtained by applying the maximum principle to the evolution equation for $t\left|\frac{\partial\curv}{\partial s}\right|^2+\curv^2$, given that $\curv$ is bounded.

\end{proof}

\section{Exponential convergence of the Normalised Flow}
\label{sec:convergence}

Now we deduce exponential convergence of the curvature to $1$.  First we
observe that since $\int\curv ds=L=2\pi$, 
\begin{align*}
  \int (\curv(s) - 1)^2 ds & = \int \curv^2 ds -2\int\curv ds +L\\
  & = \int (\curv^2-1) ds\\
  & \leq 2e^{-2(t-\bar{t})}.
\end{align*}
Stronger convergence via Gagliardo-Nirenberg inequalities (see
\cite[Theorem 19]{MR1665677}) which state that since $\int(\curv-1)ds=0$,
$$
\|D^i\curv\|_\infty\leq C(m,i)\|D^m\curv\|_\infty^{\frac{2i+1}{2m+1}}\|\curv-1\|_2^{\frac{2(m-i)}{2m+1}}
\leq C(i,\bar{t},\varepsilon)\E^{-(1-\varepsilon)t}
$$
for $t\geq 1$ and any $\varepsilon>0$, using the estimates from Corollary 1 and with $m$ chosen large enough for given $\varepsilon>0$.
Thus $\curv(s) \to 1$ in $C^\infty$ as $t \to \infty$.  It follows that the normalized curves converge modulo translations to a unit circle, exponentially fast in $C^k$ for any $k$.  The result of Grayson's theorem follows using the formulae for the unnormalized curves given in the proof of Corollary \ref{cor:LTE}.

% BIBLIOGRAPHY

\bibliographystyle{plainnat}

\end{document}